\documentclass[11pt,letterpaper,twoside,reqno,nosumlimits]{amsart}

\usepackage{fixltx2e}                       
\usepackage[usenames,dvipsnames]{xcolor}
\usepackage{fancyhdr}
\usepackage{amsmath,amsfonts,amsbsy,amsgen,amscd,mathrsfs,amssymb,amsthm}
\usepackage{subfig}
\usepackage{url}

\usepackage{mathtools}

\mathtoolsset{showonlyrefs}

\usepackage[font=small,margin=0.25in,labelfont={sc},labelsep={colon}]{caption}

\usepackage{tikz}
\usepackage{microtype}
\usepackage{enumitem}

\definecolor{dark-gray}{gray}{0.3}
\definecolor{dkgray}{rgb}{.4,.4,.4}
\definecolor{dkblue}{rgb}{0,0,.5}
\definecolor{medblue}{rgb}{0,0,.75}
\definecolor{rust}{rgb}{0.5,0.1,0.1}

\usepackage[colorlinks=true]{hyperref}

\hypersetup{urlcolor=rust}
\hypersetup{citecolor=dkblue}
\hypersetup{linkcolor=dkblue}

\usepackage{setspace}

\usepackage{graphicx}
\usepackage{booktabs,longtable,tabu} 
\usepackage{multirow} 
\usepackage{float}

\usepackage{fourier}
\usepackage{bm}

\graphicspath{{art/}}

\newtheorem{theorem}{Theorem}[section]

\newtheorem{proposition}[theorem]{Proposition}
\newtheorem{fact}[theorem]{Fact}

\newtheorem{claim}[theorem]{Claim}

\theoremstyle{definition}

\newtheorem{definition}[theorem]{Definition}

\newtheorem{remark}[theorem]{Remark}

\newcommand{\term}{\emph}

\numberwithin{equation}{section} 
\numberwithin{figure}{section}
\numberwithin{table}{section}

\floatstyle{plaintop}
\newfloat{recipe}{thp}{lor}
\floatname{recipe}{Recipe}
\numberwithin{recipe}{section}

\providecommand{\mathbold}[1]{\bm{#1}}

\renewcommand{\phi}{\varphi}

\newcommand{\cnst}[1]{\mathrm{#1}} 
\newcommand{\econst}{\mathrm{e}}

\newcommand{\Id}{\mathbf{I}}

\newcommand{\coll}[1]{\mathscr{#1}}

\providecommand{\mathbbm}{\mathbb} 
\newcommand{\R}{\mathbbm{R}}

\newcommand{\Sym}{\mathbb{S}}

\newcommand{\Prob}[1]{\mathbbm{P}\left\{{#1}\right\}}

\newcommand{\Expect}{\operatorname{\mathbb{E}}}

\newcommand{\vct}[1]{\mathbold{#1}}
\newcommand{\mtx}[1]{\mathbold{#1}}

\newcommand{\adj}{\mathsf{t}}

\newcommand{\lspan}[1]{\operatorname{span}{#1}}

\newcommand{\diag}{\operatorname{diag}}
\newcommand{\trace}{\operatorname{trace}}

\newcommand{\ip}[2]{\left\langle {#1},\ {#2} \right\rangle}

\newcommand{\norm}[1]{\left\Vert {#1} \right\Vert}
\newcommand{\normsq}[1]{\norm{#1}^2}

\newcommand{\conv}{\operatorname{conv}}

\newcommand{\maximize}{\text{maximize}\quad}
\newcommand{\subjto}{\quad\text{subject to}\quad}

\allowdisplaybreaks

\title{Simplicial Faces of the Set of Correlation Matrices}

\author[J.~A.~Tropp]{Joel~A.~Tropp}

\date{5 December 2016.  Revised 10 November 2017.}



\begin{document}

\begin{abstract}
This paper concerns the facial geometry of the set of $n \times n$
correlation matrices.  The main result states that almost every set
of $r$ vertices generates a simplicial face, provided that $r \leq \sqrt{\cnst{c} n}$,
where $\cnst{c}$ is an absolute constant.  This bound is qualitatively sharp because
the set of correlation matrices has no simplicial face generated
by more than $\sqrt{2n}$ vertices.
\end{abstract}

\maketitle

\section{Motivation}

A \term{correlation matrix} is a positive-semidefinite (psd) matrix
whose diagonal entries are identically equal to one.
The set of all correlation matrices with a fixed dimension
is called the \term{elliptope}.  As we will explain,
the elliptope arises naturally in combinatorial optimization
as an approximation to the \term{cut polytope}.
Motivated by this application, we may ask
how well the elliptope approximates the cut polytope.
In particular, it is valuable to understand
what faces the elliptope and the cut polytope have in common.
The purpose of this paper is to investigate this question.
We will demonstrate that the elliptope
and the cut polytope share an enormous
number of low-dimensional simplicial faces.

\subsection{Graphs and Cuts}

Let $G := (V, E)$ be an undirected graph with vertex set $V = \{1, \dots, n\}$.
To each subset $S$ of vertices, we associate the vector $\vct{c}_S \in \R^n$ whose entries are given by
\begin{equation} \label{eqn:cut-vector}
(\vct{c}_S)_i := \begin{cases} +1, & i \in S; \\ -1, & i \in \bar{S}. \end{cases}
\end{equation}
We have written $\bar{S} := V \setminus S$ for the set complement.
The \term{Laplacian} of the graph is the $n \times n$ psd matrix
\begin{equation} \label{eqn:laplacian}
\mtx{L}
	:= \mtx{L}_G
	:= \frac{1}{4} \sum\limits_{\{i, j\} \in E} (\mathbf{e}_i - \mathbf{e}_j) ( \mathbf{e}_i - \mathbf{e}_j)^\adj
\end{equation}
where the vector $\mathbf{e}_i \in \R^n$ has a one in the $i$th coordinate and zeros elsewhere.
For a subset $S$ of vertices, we can easily evaluate the quadratic form defined by the Laplacian $\mtx{L}$
at the vector $\vct{c}_S$:
\begin{equation} \label{eqn:dirichlet-form}
{\vct{c}_S}^\adj \mtx{L} \vct{c}_S
	= \frac{1}{4} \sum\limits_{\{i,j\} \in E} \big((\vct{c}_S)_i - (\vct{c}_S)_j \big)^2
	= \#\big\{ (i, j) \in S \times \bar{S} : \{ i, j \} \in E \big\}.
\end{equation}
In words, the value ${\vct{c}_S}^\adj \mtx{L} \vct{c}_S$ of the quadratic form
equals the number of edges that connect $S$ and its complement $\bar{S}$,
which is called the \term{weight} of the graph cut induced by $S$.

\subsection{Combinatorial and Semidefinite Formulations of the Maximum Cut}

This discussion suggests that we can use a mathematical program to optimize
the weight of a cut.  Let us introduce some definitions.

\begin{definition}[Cuts] Let $n$ be a natural number.  An $n$-dimensional \term{cut vector} is a member of
set $\{ \pm 1 \}^n$.  An $n \times n$ \emph{cut matrix} takes the form
$\vct{cc}^\adj$ where $\vct{c}$ is an $n$-dimensional cut vector.
The \emph{cut polytope} $\coll{C}_n$ is the convex hull of the $n \times n$ cut matrices:
\begin{equation} \label{eqn:cut-polytope}
\coll{C}_n := \conv\big\{ \vct{cc}^\adj : \vct{c} \in \{\pm 1\}^n \big\}.
\end{equation}
These objects are sometimes called \term{signed} cut vectors, matrices, and polytopes.
\end{definition}

Let $\mtx{A}$ be an $n \times n$ real psd matrix,
and consider three equivalent mathematical programs:
\begin{equation} \label{eqn:maxcut}
\begin{aligned}
\maximize & \vct{x}^\adj \mtx{A} \vct{x}
&&\subjto\quad
\text{$\vct{x}$ is a cut vector.} \\
\maximize & \trace(\mtx{A}\mtx{X})
&&\subjto\quad
\text{$\mtx{X}$ is a cut matrix.} \\
\maximize & \trace( \mtx{A} \mtx{X} )
&&\subjto\quad
\text{$\mtx{X} \in \coll{C}_n$.} 
\end{aligned}
\end{equation}
To see that the first two are equivalent,
write $\mtx{X} = \vct{xx}^\adj$ for a cut vector $\vct{x}$
and cycle the trace.
As for the third,
the construction~\eqref{eqn:cut-polytope} implies that
the extreme points of the cut polytope $\coll{C}_n$
are precisely the $n \times n$ cut matrices.
Therefore, the third program attains its
optimal value at a cut matrix.

In view of~\eqref{eqn:dirichlet-form},
we can try to find the maximum weight of a cut in the graph $G$ by
solving~\eqref{eqn:maxcut} with $\mtx{A} = \mtx{L}_G$,
where $\mtx{L}_G$ is the graph Laplacian~\eqref{eqn:laplacian}.
Finding the maximum weight of a cut in a general graph is
\textsf{NP}-hard~\cite{Kar72:Reducibility-Combinatorial},
so we cannot accomplish this task with a polynomial-time algorithm
unless $\mathsf{P} = \mathsf{NP}$.

One remedy is to relax~\eqref{eqn:maxcut} to reach
a tractable computational problem.  To do so, notice that each
cut matrix is a real psd matrix whose diagonal
entries are equal to one.  This motivates another definition.

\begin{definition}[Elliptope]
Let $n$ be a natural number.  The \term{elliptope} $\coll{E}_n$ is the convex set
\begin{equation} \label{eqn:elliptope}
\coll{E}_n := \big\{ \mtx{X} \in \Sym_+^n : \diag(\mtx{X}) = \mathbf{1} \big\}.
\end{equation}
The set $\Sym_+^n$ comprises the $n \times n$ real psd matrices,
and each entry of the vector $\vct{1} \in \R^n$ equals one.
The members of the elliptope are called \term{correlation matrices}.
\end{definition}

The elliptope $\coll{E}_n$ contains every $n \times n$ cut matrix,
so it also contains the cut polytope $\coll{C}_n$.  
Therefore, we may attempt to approximate
the value of~\eqref{eqn:maxcut} by means
of the semidefinite programming problem
\begin{equation} \label{eqn:maxcut-sdp}
\maximize \trace( \mtx{A} \mtx{X} )
\quad\subjto\quad
\text{$\mtx{X} \in \coll{E}_n$.}
\end{equation}
The elliptope $\coll{E}_n$ is an affine slice of the psd cone $\Sym_+^n$,
so there are polynomial-time algorithms for
completing the optimization~\eqref{eqn:maxcut-sdp}
to a fixed accuracy in the real arithmetic model~\cite[Lec.~5]{BN01:Lectures-Modern}.

\subsection{Analysis of the Semidefinite Relaxation}

How well does the tractable formulation~\eqref{eqn:maxcut-sdp} work?
For psd $\mtx{A}$,
a randomized rounding argument~\cite{Nes98:Semidefinite-Relaxation}
shows that the optimal values of~\eqref{eqn:maxcut} and~\eqref{eqn:maxcut-sdp} satisfy
\begin{equation} \label{eqn:maxcut-approx}
\frac{2}{\pi} \cdot \operatorname{val} \eqref{eqn:maxcut-sdp}
	\leq \operatorname{val} \eqref{eqn:maxcut}
	\leq \operatorname{val} \eqref{eqn:maxcut-sdp}.
\end{equation}
The constant $2/\pi$ cannot be improved.
But it has been observed empirically that
the value of ~\eqref{eqn:maxcut-sdp} is
usually within a few percent of the value of~\eqref{eqn:maxcut}.  For example,
see~\cite{DP93:Performance-Eigenvalue,PR95:Solving-Max-Cut,GW95:Improved-Approximation,MT11:Two-Proposals}.

Results like~\eqref{eqn:maxcut-approx}
are often attributed to Goemans \& Williamson~\cite{GW95:Improved-Approximation}
or to Nesterov~\cite{Nes98:Semidefinite-Relaxation}, but the provenance is longer.
Indeed, the bound~\eqref{eqn:maxcut-approx}
is equivalent to the ``little'' Grothendieck theorem~\cite[Thm.~4]{Gro53:Resume-Theorie}.
See the surveys~\cite{KN12:Grothendieck-Type-Inequalities,Pis12:Grothendiecks-Theorem}
for a modern introduction to Grothendieck's work.

\begin{figure}[t]
\begin{center}
\includegraphics[width=0.48\textwidth]{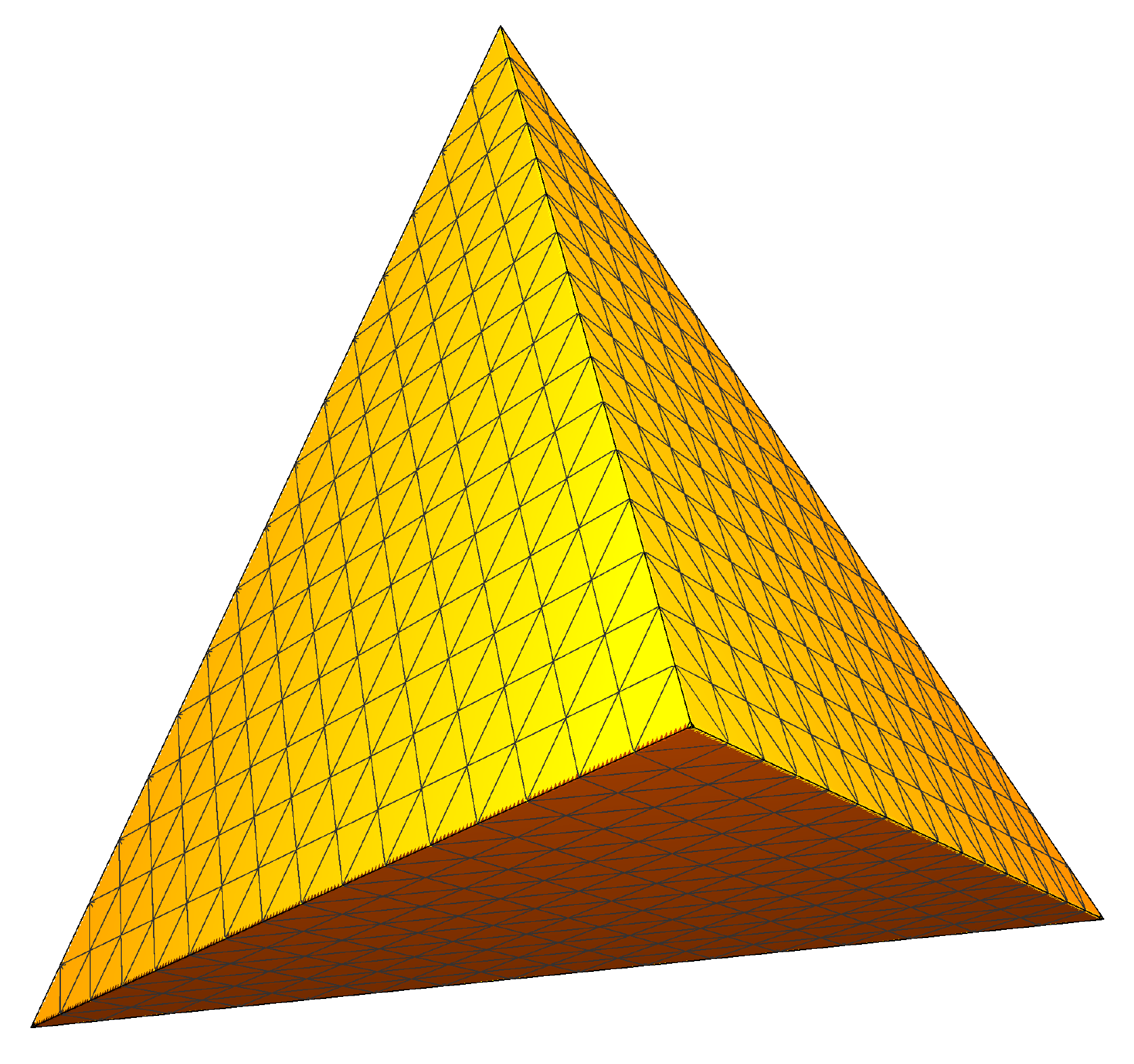} \hfill
\includegraphics[width=0.48\textwidth]{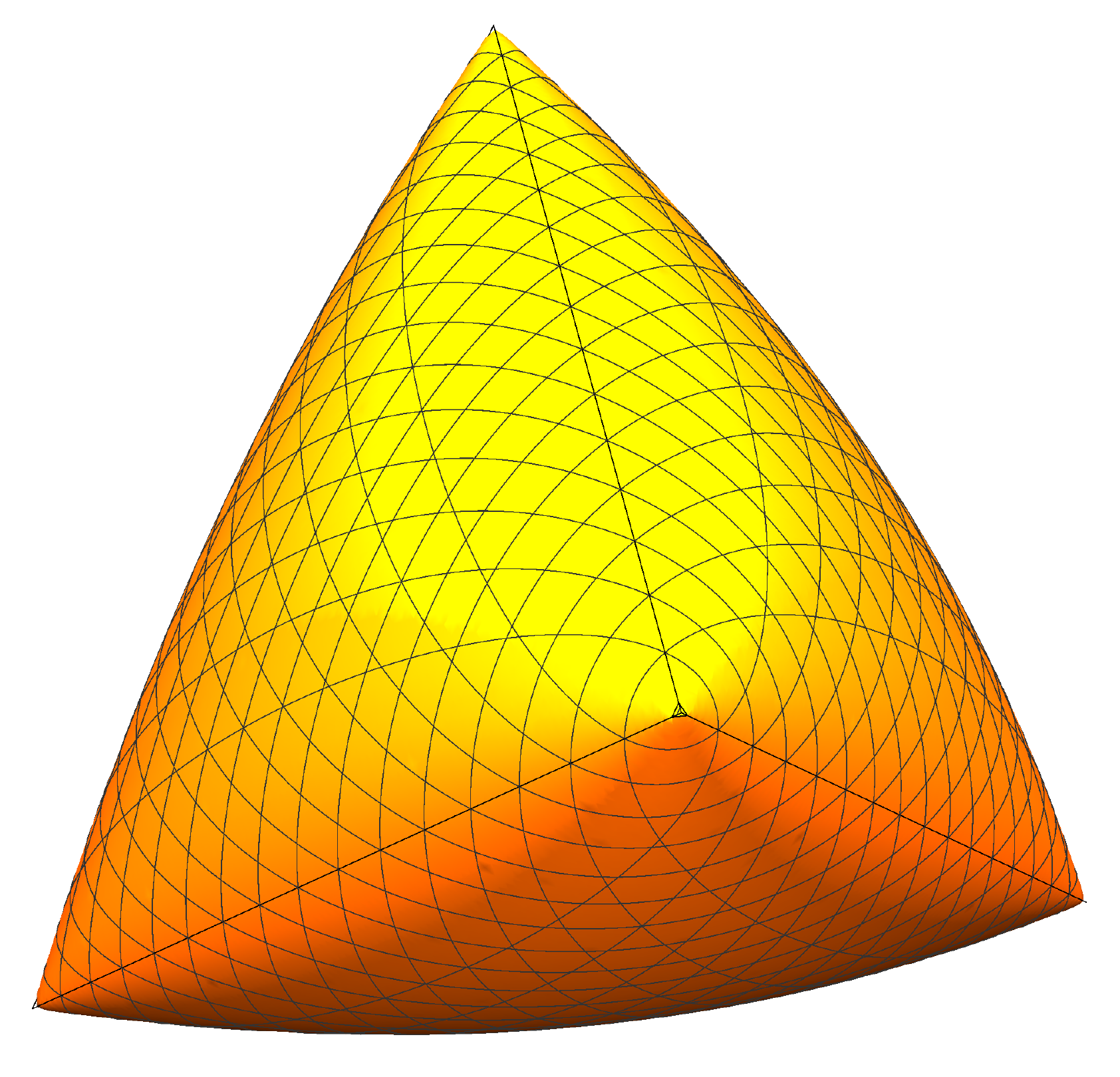}
\caption{\textsl{The Cut Polytope and the Elliptope.}
Define the linear map $\mtx{T} : \Sym^3 \to \R^3$ that extracts
the strict lower triangle of a $3 \times 3$ symmetric matrix
as a 3-dimensional vector.
\textbf{[left]} The image $\mtx{T}(\coll{C}_3)$ of the cut polytope $\coll{C}_3$.
The closest vertex is the image of the cut matrix $\vct{11}^\adj \in \Sym^3_+$.
\textbf{[right]} The image $\mtx{T}(\coll{E}_3)$ of the elliptope $\coll{E}_3$,
seen from the same vantage.}
\label{fig:elliptope}
\end{center}
\end{figure}

\subsection{Facial Geometry of the Elliptope and Optimization}

From the bound~\eqref{eqn:maxcut-approx}, we learn that the elliptope $\coll{E}_n$ is
a uniformly good relaxation of the cut polytope $\coll{C}_n$.
Yet the elliptope approximates the cut polytope
far more accurately than~\eqref{eqn:maxcut-approx} suggests.

Indeed, it is fruitful to think of the elliptope as a
``shrink-wrapped'' cut polytope.
By this, we mean that the elliptope $\coll{E}_n$ adheres to low-dimensional faces of the cut polytope $\coll{C}_n$,
while it curves away from the higher-dimensional faces.
Figure~\ref{fig:elliptope} illustrates this claim
for $\coll{C}_3$ and $\coll{E}_3$.
We see that both $\coll{C}_3$ and $\coll{E}_3$ have the same vertices and edges,
while the facets of $\coll{C}_3$ are not preserved
in the relaxation $\coll{E}_3$.

This paper demonstrates that a similar phenomenon holds more broadly.
We will prove that many low-dimensional simplicial faces of the cut polytope $\coll{C}_n$ are also faces of the elliptope $\coll{E}_n$.

This geometric observation provides a heuristic understanding of why the semidefinite
relaxation~\eqref{eqn:maxcut-sdp} often performs better than~\eqref{eqn:maxcut-approx} suggests.

\subsection{Notation}

We write $\norm{\cdot}$ for the Euclidean norm on $\R^d$.
The standard basis vector $\mathbf{e}_i$ has a one in the
$i$th coordinate and zeros elsewhere.  The symbol $\vct{1}$
refers to a vector whose entries are identically equal to
one.  The dimensions of these special vectors are
determined by context.
The notation ${}^\adj$ represents the transpose of a vector.
We frequently use the componentwise product $\odot$ of two vectors,
which is also known as the Hadamard or Schur product.

The set $\Sym^n$ contains the $n \times n$ real symmetric matrices,
and $\Sym_+^n$ is the subset of $n \times n$ real psd matrices.
The map $\lambda_{\min} : \Sym^n \to \R$ computes
the smallest eigenvalue of a symmetric matrix.
The letter $\Id$ refers to the identity matrix,
and the letter $\mathbf{J} := \vct{11}^\adj$ denotes a square matrix of ones.
The dimensions of these special matrices are determined by context. The symbol $\vee$ refers to the symmetric tensor product
of vectors or matrices; see~\cite[Chap.~I]{Bha97:Matrix-Analysis}
for an overview of multilinear algebra.

The operator $\mathbb{P}$ returns the probability of an event,
while $\Expect$ computes the expectation of a random variable,
a random vector, or a random matrix.  The abbreviation \term{iid}
means independent and identically distributed.
Small capitals (e.g., \textsc{SBern})
are used for the names of probability distributions.
The symbol $\sim$ means ``has the distribution.''

\section{Background and Results}

This section outlines our results
on the facial structure of the elliptope~\eqref{eqn:elliptope}.
We begin with a review of the definition of a simplicial face of a convex set,
and we summarize known results about the simplicial faces of the elliptope.
Next, we describe a random model that generates candidates for simplicial faces.
The main results delineate situations where this construction is likely to be successful.

\subsection{Facial Geometry of Convex Sets}

We begin with a reminder about some relevant definitions from convex geometry.
For further background, see~\cite{Roc70:Convex-Analysis,HL01:Fundamentals-Convex}.

\begin{definition}[Dimension]
Let $K$ be a convex set in $\R^d$.
The \term{dimension} of $K$ is defined as the dimension of the affine hull of $K$.
\end{definition}

\begin{definition}[Face]
Let $K$ be a convex set in $\R^d$.
A \term{face} $F$ of $K$ is a convex subset of $K$ for which
$$
\vct{x}, \vct{y} \in K
\quad\text{and}\quad
\text{$\theta \vct{x} + (1-\theta)\vct{y} \in F$
for some $\theta \in (0, 1)$}
\quad\text{imply that}\quad
\vct{x}, \vct{y} \in F.
$$
A face is also called an \term{extreme set}.
A 0-dimensional face is commonly called an \term{extreme point}.
\end{definition}

\begin{definition}[Simplicial Face]
A $k$-dimensional face $F$ of a convex set is \term{simplicial}
if $F$ is the convex hull of an affinely independent family
of $k + 1$ points.
\end{definition}

\begin{definition}[Vertex]
Let $K$ be a convex set in $\R^d$.  A point $\vct{x} \in K$
is a \term{vertex} of $K$ if the normal cone $\coll{N}(\vct{x}; K)$
has dimension $d$.  For reference, $\coll{N}(\vct{x}; K) := \{ \vct{z} \in \R^d  : \vct{z}^\adj( \smash{\vct{y}} - \vct{x} ) \leq 0
\text{ for all $\vct{y} \in K$} \}$.
\end{definition}

Heuristically, a vertex is a sharp corner of a convex set.
Vertices are always extreme points, but extreme points need not be vertices!

\subsection{Facial Geometry of the Elliptope}

The literature contains a lot of information about the
facial geometry of the elliptope. Let us present some key results,
which are due to Laurent \& Poljak~\cite{LP95:Positive-Semidefinite-Relaxation,LP96:Facial-Structure}.

\begin{fact}[Vertices] \label{fact:vertices}
The elliptope $\coll{E}_n$ has $2^{n-1}$ vertices.
These vertices are precisely the $n \times n$ cut matrices
$\vct{cc}^\adj$, where $\vct{c} \in \{ \pm 1 \}^n$ is a cut vector.
\end{fact}

The most natural candidate for a face of the elliptope
is the convex hull of a set of vertices.  This
construction does not always yield a face,
but---when it does---that face is always simplicial.  

\begin{fact}[Faces Generated by Vertices are Simplicial] \label{fact:simp-form}
Let $\vct{c}_1, \dots, \vct{c}_r \in \{ \pm 1 \}^n$ be cut vectors.
Consider the set
$$
F := \conv \big\{ \vct{c}_1 {\vct{c}_1}^\adj, \dots, \vct{c}_r {\vct{c}_r}^\adj \big\} \subset \coll{E}_n.
$$
If $F$ is a face of the elliptope $\coll{E}_n$, then $F$ is a simplicial face of $\coll{E}_n$.
\end{fact}

Fact~\ref{fact:simp-form} does \emph{not} assert that every simplicial
face of the elliptope is generated by vertices.  Even so,
we can bound the possible dimension of a simplicial face,
regardless of its structure.

\begin{fact}[Dimension of Simplicial Faces] \label{fact:simp-dim}
The elliptope $\coll{E}_n$ has a simplicial face of dimension $k$
if and only if $k(k+1) \leq 2(n-1)$.
In particular, it is necessary that the dimension $k < \sqrt{2(n-1)}$.
\end{fact}

Our interest in the facial structure of the elliptope
is motivated by its connection with the facial structure
of the cut polytope~\eqref{eqn:cut-polytope}.

\begin{fact}[Coincidental Faces] \label{fact:coincidence}
If $F$ is a face of the elliptope $\coll{E}_n$ generated by vertices,
then $F$ is also a face of the cut polytope $\coll{C}_n$.
\end{fact}

\noindent
Indeed, Fact~\ref{fact:coincidence} follows directly from the definition
of a face, the fact that the elliptope contains
the cut polytope, and the fact that the vertices of the elliptope are
elements of the cut polytope.

\subsection{A Random Model for Simplicial Faces}

We will study the extent to which the elliptope $\coll{E}_n$
approximates the cut polytope $\coll{C}_n$ by identifying a large number
of simplicial faces of the elliptope.
Our approach is based on the probabilistic method.
In view of Fact~\ref{fact:simp-form}, we can attempt
to construct simplicial faces of $\coll{E}_n$ by drawing
a collection of random vertices and forming its convex hull.

For a parameter $p \in [0,1]$,
we define the \term{signed Bernoulli distribution}:
$$
\textsc{SBern}(p) := \begin{cases}
	+1, & \text{with probability $p$;} \\
	-1, & \text{with probability $1-p$.}
\end{cases}
$$
We extend this distribution to vectors of length $n$
by taking a direct product:
$$
\textsc{SBern}(p, n) := \textsc{SBern}(p) \times \dots \times \textsc{SBern}(p)
	\in \{ \pm 1 \}^n.
$$
That is, a random vector from $\textsc{SBern}(p,n)$
has $n$ entries, each drawn independently from $\textsc{SBern}(p)$.

Fix the parameter $p \in (0, 1)$,
the number $r$ of vertices, and the dimension $n$.
Let us present a random model $\textsc{Face}(p,r,n)$ for
a prospective face $F$ of the elliptope $\coll{E}_n$.
Draw random vectors $\vct{\xi}_1, \dots, \vct{\xi}_r$
independently from the distribution $\textsc{SBern}(p,n)$.
Construct the random convex set
\begin{equation} \label{eqn:random-face}
F 	:= \conv\big\{ \vct{\xi}_1 {\vct{\xi}_1}^\adj, \dots, \vct{\xi}_r {\vct{\xi}_r}^\adj \big\}
	\subset \coll{E}_n.
\end{equation}
Our goal is to understand when $F$ is likely to be a simplicial face of $\coll{E}_n$.
The parameter $p$ controls the typical ``balance'' of positive and
negative entries that appear in the random vertices
$\vct{\xi}_i{\vct{\xi}_i}^\adj$.
See Figure~\ref{fig:balance} for a simple illustration.

\begin{figure}[t]
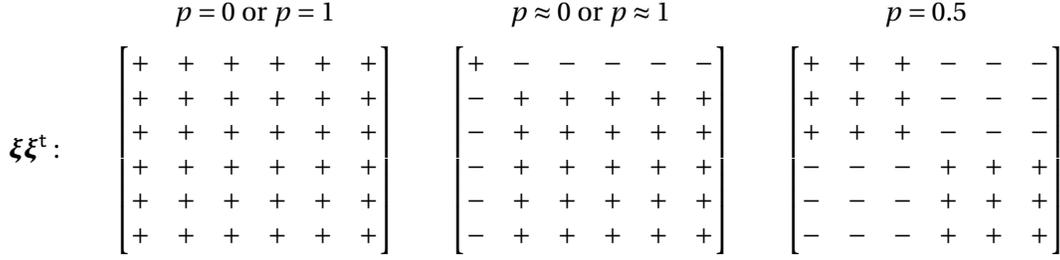

\begin{center}
$$
\begin{array}{ccccccc}
&& p = 0 \text{ or } p = 1
&&
p \approx 0 \text{ or } p \approx 1
&&
p = 0.5
\vspace{0.5pc} \\
\vct{\xi \xi}^\adj :
&\hspace{1pc} &
\begin{bmatrix}
	+ & + & + & + & + & + \\
	+ & + & + & + & + & + \\
	+ & + & + & + & + & + \\
	+ & + & + & + & + & + \\
	+ & + & + & + & + & + \\
	+ & + & + & + & + & +
\end{bmatrix}
& \hspace{3pc} &
\begin{bmatrix}
	+ & - & - & - & - & - \\
	- & + & + & + & + & + \\
	- & + & + & + & + & + \\
	- & + & + & + & + & + \\
	- & + & + & + & + & + \\
	- & + & + & + & + & + \\
\end{bmatrix}
& \hspace{3pc} &
\begin{bmatrix}
	+ & + & + & - & - & - \\
	+ & + & + & - & - & - \\
	+ & + & + & - & - & - \\
	- & - & - & + & + & + \\
	- & - & - & + & + & + \\
	- & - & - & + & + & +
\end{bmatrix}
\end{array}
$$
\caption{\textsl{Balance of Random Cut Matrices.}
Let $\vct{\xi} \sim \textsc{SBern}(p,n)$.
The balance parameter $p$ controls the (average) proportion
of entries in the random cut matrix $\vct{\xi \xi}^\adj$
that are positive and negative.  This display gives a nominal illustration of the effect.
We have abbreviated the numbers $\pm 1$ by their signs $\pm$.}
\label{fig:balance}
\end{center}
\end{figure}

\subsection{Simplicial Faces: Quantitative Results}

Our first set of results gives quantitative bounds on
the probability that the random model $\textsc{Face}(p,r,n)$
generates a simplicial face of $\coll{E}_n$.  By this,
we mean that the bounds contain explicit and reasonable
constants, but the form of the bounds is suboptimal.
The first statement concerns the special case where the
balance parameter $p = 0.5$.

\begin{theorem}[Simplicial Faces I] \label{thm:balanced}
Fix the balance parameter $p = 0.5$, the number $r$ of vertices, and the dimension $n$.
Draw a random set $F$ from the distribution $\emph{\textsc{Face}}(p, r, n)$
described in~\eqref{eqn:random-face}.
Then
$$
\Prob{ \text{$F$ is a simplicial face of $\coll{E}_n$ with dimension $r-1$} }
	\quad\geq\quad 1 \ -\ r^2 \exp\left( \frac{-n}{r^2} \right).
$$
\end{theorem}

\noindent
The proof of Theorem~\ref{thm:balanced} begins in Section~\ref{sec:proof-overview}
and continues in Section~\ref{sec:chernoff}.

Here is the most noteworthy consequence of Theorem~\ref{thm:balanced}.
When $r^2 \log r^2 < n$, there is a positive probability that a random set
$F \sim \textsc{Face}(0.5,r,n)$ is a simplicial face of $\coll{E}_n$ with
dimension $r - 1$.  The stricter bound $r \leq \sqrt{n/\log n}$ is also
sufficient for $F$ to be a simplicial face.
It follows from Fact~\ref{fact:coincidence} that $\coll{E}_n$ and $\coll{C}_n$
share a large number of simplicial faces with dimension up to $\sqrt{n / \log n}$.

Next, we generalize to the case where the balance parameter $p$ is general.
This result has slightly larger constants than Theorem~\ref{thm:balanced}.

\begin{theorem}[Simplicial Faces II] \label{thm:unbalanced-chernoff}
Fix the balance parameter $p \in (0,1)$, the number $r$ of vertices, and the dimension $n$.
Draw a random set $F$ from the distribution $\emph{\textsc{Face}}(p, r, n)$
described in~\eqref{eqn:random-face}.
Then
$$
\Prob{ \text{$F$ is a simplicial face of $\coll{E}_n$ with dimension $r-1$} }
	\quad\geq\quad 1 \ -\ r^2 \exp\left( \frac{-4p^2(1-p)^2 n}{r^2} \right).
$$
\end{theorem}

\noindent
The proof of Theorem~\ref{thm:unbalanced-chernoff} begins in Section~\ref{sec:proof-overview}
and continues in Section~\ref{sec:chernoff}.

Theorem~\ref{thm:unbalanced-chernoff} yields simplicial faces when the number $r$
satisfies $r \leq 2p(1-p) \sqrt{n/\log n}$.  In particular,
for any fixed choice of the balance parameter $p$ and the number $r$ of vertices,
the model $\textsc{Face}(p,r,n)$ produces simplicial faces whenever the dimension $n$ is sufficiently large.

\subsection{Simplicial Faces: Qualitative Result}

Our last result gives a qualitative bound on the probability that
the random model $\textsc{Face}(p,r,n)$ generates a simplicial face of $\coll{E}_n$.
By this, we mean that the form of the bound improves on Theorem~\ref{thm:unbalanced-chernoff},
but the analysis incurs large constant factors.

\begin{theorem}[Simplicial Faces III] \label{thm:unbalanced}
Fix the balance parameter $p \in (0,1)$, the number $r$ of vertices, and the dimension $n$.
Draw a random set $F$ from the distribution $\emph{\textsc{Face}}(p, r, n)$
described in~\eqref{eqn:random-face}.
Then
$$
\Prob{ \text{$F$ is a simplicial face of $\coll{E}_n$ with dimension $r-1$} }
	\quad\geq\quad 1 \ -\ 4 \exp\left( \frac{r^2 - \cnst{c} p^2(1-p)^2 n}{4} \right).
$$
The number $\cnst{c}$ is a positive, absolute constant that satisfies $\cnst{c} \geq 0.0003$.
\end{theorem}

\noindent
The proof of Theorem~\ref{thm:unbalanced} begins in Section~\ref{sec:proof-overview}
and continues in Section~\ref{sec:oliveira}.

Here is the key consequence of Theorem~\ref{thm:unbalanced}.
When $r + 3 \leq p(1-p) \sqrt{\cnst{c} n}$, there is a positive probability that
a random set $F \sim \textsc{Face}(p,r,n)$ is a simplicial face of $\coll{E}_n$
with dimension $r - 1$.  Fact~\ref{fact:simp-dim} shows that this bound is
qualitatively optimal when $p = 0.5$.

We see that Theorem~\ref{thm:unbalanced} removes a parasitic logarithmic
term from Theorem~\ref{thm:balanced}.  In addition, the probability bound
in Theorem~\ref{thm:unbalanced} is significantly stronger.  On the other
hand, the relative size of the constants ensures that Theorem~\ref{thm:balanced}
gives a quantitative benefit for any realistic dimension $n$.

\subsection{Counting Faces}

Theorems~\ref{thm:balanced},~\ref{thm:unbalanced-chernoff},
and~\ref{thm:unbalanced} also have an enumerative interpretation.
Draw a random vector $\vct{\xi} \sim \textsc{SBern}(0.5, n)$.
Then $\vct{\xi}$ is uniformly distributed over the set $\{ \pm 1 \}^n$
of cut vectors, and Fact~\ref{fact:vertices} implies that $\vct{\xi\xi}^\adj$
is a uniformly random vertex of the elliptope $\coll{E}_n$.
These observations yield an alternative procedure for
drawing a random set $F$ from the model $\textsc{Face}(0.5,r,n)$:
Let $F$ be the convex hull of $r$ vertices of $\coll{E}_n$,
chosen uniformly at random, with replacement.\footnote{In our parameter regime, it is unlikely that any vertex of $\coll{E}_n$
is chosen more than once, so this model is not substantially different from drawing
vertices without replacement.}
We obtain roughly $( \econst 2^{n-1} / r )^r$
different sets in this manner.
When $r \ll \sqrt{n}$, most of these sets are simplicial faces of $\coll{E}_n$.

\subsection{Related Work}

We can articulate the heuristic that a ``small'' collection of ``generic''
vertices of the elliptope induces a simplicial face.
Theorems~\ref{thm:balanced},~\ref{thm:unbalanced-chernoff}, and~\ref{thm:unbalanced}
are all instantiations of this principle.
In fact, this idea is already visible in a combinatorial construction
of Laurent \& Poljak~\cite[Cor.~4.5]{LP96:Facial-Structure}.

\begin{fact}[Combinatorial Construction of Simplicial Faces] \label{fact:comb-face}
Let $\vct{c}_1, \dots, \vct{c}_r \in \{ \pm 1 \}^n$
be cut vectors in \term{general position}.  That is,
\begin{equation*} \label{eqn:general-position}
\bigg[ \bigodot_{i \in I} (\vct{1} + \vct{c}_i) \bigg] \odot
\bigg[ \bigodot_{i \notin I} (\vct{1} - \vct{c}_i) \bigg]
\neq \vct{0}
\quad\text{for each subset $I$ of $\{1, \dots, r\}$.}
\end{equation*}
Then $
\conv\{ \vct{c}_1{\vct{c}_1}^\adj, \dots, \vct{c}_r {\vct{c}_r}^\adj \}
$
is a simplicial face of the elliptope $\coll{E}_n$.
\end{fact}

Fact~\ref{fact:comb-face} only has content when the number $r$ of cut vectors
satisfies $r \leq  \log_2 n $.  In contrast,
our probabilistic argument is valid in a wider parameter range.
Theorem~\ref{thm:balanced} operates in the regime $r \leq \sqrt{n/\log n}$,
and Theorem~\ref{thm:unbalanced} has implications when $r \leq \sqrt{\cnst{c} n}$.
Even so, our analysis depends on the same sufficient condition
(Fact~\ref{fact:lp-simplicial}, below)
that Laurent \& Poljak use to establish Fact~\ref{fact:comb-face}.

\section{Proof Strategy}
\label{sec:proof-overview}

This section outlines our technique for proving
Theorems~\ref{thm:balanced},~\ref{thm:unbalanced-chernoff}, and~\ref{thm:unbalanced}.
The argument begins with a sufficient condition,
due to Laurent \& Poljak~\cite{LP96:Facial-Structure},
for a family of cut vectors to generate a simplicial face.
The challenge is to understand the probability
that a collection of \emph{random} cut vectors satisfies the sufficient condition.
We explain how to reduce this question to a problem
that can be addressed using matrix concentration
inequalities.  In Sections~\ref{sec:chernoff} and~\ref{sec:oliveira},
we carry out this program.

\subsection{Deterministic Condition for a Simplicial Face}

The first ingredient in our argument is a sufficient
condition~\cite[Thm.~4.2]{LP96:Facial-Structure} for a family of cut vectors to generate
a simplicial face of the elliptope.

\begin{fact}[Sufficient Condition for a Simplicial Face] \label{fact:lp-simplicial}
Assume that $r \geq 2$.  Let $\vct{c}_1, \dots, \vct{c}_r \in \{ \pm 1 \}^n$ be cut vectors,
and consider the set
$$
F := \conv\big\{ \vct{c}_1 {\vct{c}_1}^\adj, \dots, \vct{c}_r {\vct{c}_r}^\adj \big\}
	\subset \coll{E}_n.
$$
Define $R := 1 + r(r-1)/2$, and introduce the two matrices
\begin{equation} \label{eqn:WZ}
\begin{aligned}
\mtx{W} := \mtx{W}(\vct{c}_1, \dots, \vct{c}_r) &:= \left[ \begin{array}{ccc} \vct{c}_1 & \dots & \vct{c}_r \end{array}\right] \in \R^{n \times r}; \\
\mtx{Z} := \mtx{Z}(\vct{c}_1, \dots, \vct{c}_r) &:= \left[ \begin{array}{c|ccccc} \vct{1}
	& \vct{c}_1 \odot \vct{c}_2 & \dots & \vct{c}_i \odot \vct{c}_j & \dots & \vct{c}_{r-1} \odot \vct{c}_r
	\end{array} \right] \in \R^{n \times R}
	\quad\text{where $1 \leq i < j \leq r$.}
\end{aligned}
\end{equation}
If $\mtx{W}$ and $\mtx{Z}$ both have full column rank,
then $F$ is a simplicial face of $\coll{E}_n$
with dimension $r-1$.
\end{fact}

For our purposes, it is more natural to consider the dual form of the condition
that $\mtx{W}$ and $\mtx{Z}$ have full column rank.  Express these two matrices
in terms of their rows:
$$
\mtx{W} = \begin{bmatrix} {\vct{w}_1}^\adj \\ \vdots \\ {\vct{w}_n}^\adj \end{bmatrix} \in \R^{n \times r}
\quad\text{and}\quad
\mtx{Z} = \begin{bmatrix} {\vct{z}_1}^\adj \\ \vdots \\ {\vct{z}_n}^\adj \end{bmatrix} \in \R^{n \times R}.
$$
Suppose that
\begin{equation} \label{eqn:lp-dual}
\lspan\{ \vct{w}_1, \dots, \vct{w}_n \} = \R^r
\quad\text{and}\quad
\lspan\{ \vct{z}_1, \dots, \vct{z}_n \} = \R^R.
\end{equation}
Then Fact~\ref{fact:lp-simplicial} implies that the $F$ is a simplicial face of $\coll{E}_n$.

\begin{remark}[Variant Sufficient Condition] \label{rem:var-suff}
Let $R' := r(r-1)/2$.
Using the same notation as in Fact~\ref{fact:lp-simplicial}, we define the matrix
\begin{equation*} \label{eqn:Y}
\mtx{Y} := \mtx{Y}(\vct{c}_1, \dots, \vct{c}_r) := \left[ \begin{array}{ccccc}
	\vct{1} - \vct{c}_1 \odot \vct{c}_2 & \dots &
	\vct{1} - \vct{c}_i \odot \vct{c}_j & \dots &
	\vct{1} - \vct{c}_{r-1} \odot \vct{c}_r \end{array} \right] \in \R^{n \times R'}.
\end{equation*}
The indices lie in the range $1 \leq i < j \leq r$.
Using the fact that the $\vct{c}_i$ are cut vectors,
Laurent \& Poljak~\cite[Condition (iii), p.~540]{LP96:Facial-Structure} show that $\mtx{Y}$
has full column rank if and only if $\mtx{Z}$ has full column rank.
\end{remark}

\subsection{Sufficient Condition for the Random Model}
\label{sec:suff-rnd}

Fix the balance parameter $p \in (0,1)$, the number $r$ of vertices,
and the ambient dimension $n$.  Draw independent random vectors
$\vct{\xi}_1, \dots, \vct{\xi}_r$ from the distribution $\textsc{SBern}(p, n)$,
and construct the random set
$$
F := \conv\big\{ \vct{\xi}_1{\vct{\xi}_1}^\adj, \dots, \vct{\xi}_r {\vct{\xi}_r}^\adj \big\}
	\sim \textsc{Face}(p,r,n).
$$
We need to determine the probability that $\vct{\xi}_1, \dots, \vct{\xi}_r$
satisfy the sufficient condition from Fact~\ref{fact:lp-simplicial}.
This gives a lower bound on the probability that $F$ is a simplicial face of $\coll{E}_n$.

We can check the dual form~\eqref{eqn:lp-dual} of the sufficient condition.
Consider the matrix $\mtx{W} = \mtx{W}(\vct{\xi}_1, \dots, \vct{\xi}_r) \in \R^{n \times r}$,
defined in~\eqref{eqn:WZ}.
Observe that the coordinates of the $\vct{\xi}_i$ are iid,
so the matrix $\mtx{W}$ has iid rows.
More precisely, the $n$ rows of $\mtx{W}$ are iid copies of a random vector
$\vct{w} \in \R^r$ where $\vct{w} \sim \textsc{SBern}(p, r)$.

In a similar vein, consider the matrix $\mtx{Z} = \mtx{Z}(\vct{\xi}_1, \dots, \vct{\xi}_r) \in \R^{n\times R}$,
defined in~\eqref{eqn:WZ}.
Introduce a random vector $\vct{z} \in \R^R$, whose entries are derived
from the random vector $\vct{w} \in \R^r$ as follows.
\begin{equation} \label{eqn:z-vec}
z_0 := 1
\quad\text{and}\quad
z_{ij} := w_i w_j
\quad\text{for $1 \leq i < j \leq r$.}
\end{equation}
Then the $n$ rows of $\mtx{Z}$ are iid copies of the random vector $\vct{z}$.

Therefore, to verify~\eqref{eqn:lp-dual},
we must compute the probability that $n$ iid copies of the random vector $\vct{w}$
span the space $\R^r$
and that $n$ iid copies of the random vector $\vct{z}$
span the space $\R^R$.
The following proposition summarizes this discussion.

\begin{proposition}[Sufficient Condition for Random Model] \label{prop:suff-rdm}
Fix the balance parameter $p \in (0,1)$, the number $r$ of vertices where $r \geq 2$,
and the dimension $n$.  Define $R := 1 + r(r-1)/2$.  Introduce a random vector
$\vct{w} \sim \textsc{SBern}(p,r)$, and define $\vct{z} \in \R^R$ by the formula~\eqref{eqn:z-vec}.

Draw iid copies $\vct{w}_1, \dots, \vct{w}_n$ of the random vector $\vct{w} \in \R^r$
and iid copies $\vct{z}_1, \dots, \vct{z}_n$ of the random vector $\vct{z} \in \R^R$.
Then a random set $F \sim \textsc{Face}(p,r,n)$ satisfies
$$
\Prob{ \text{$F$ is \emph{not} a simplicial face of $\coll{E}_n$} }
	\quad\leq\quad \Prob{ \lspan\{ \vct{w}_1, \dots, \vct{w}_n \} \neq \R^r }
	\ + \ \Prob{ \lspan\{ \vct{z}_1, \dots, \vct{z}_n \} \neq \R^R }.
$$
\end{proposition}

The easiest way to complete the calculations required
by Proposition~\ref{prop:suff-rdm} is to invoke
methods from the field of matrix concentration
inequalities~\cite{Tro15:Introduction-Matrix}.
Among other things, this theory gives practical estimates
for the minimum singular value of a random matrix with iid rows.
This type of result leads directly to a bound on the probability
that an iid family of random vectors spans a linear space.
In the next two sections, we complete our program by
combining Proposition~\ref{prop:suff-rdm} with
two different types of matrix concentration.

\begin{remark}[Variant Sufficient Condition for Random Model] \label{rem:var-suff-rdm}
Instate the notation from Proposition~\ref{prop:suff-rdm}.
Define $R' := r(r-1)/2$, and derive a random vector $\vct{y} \in \R^{R'}$
from the vector $\vct{w} \in \R^r$ as follows.
\begin{equation} \label{eqn:y-vec}
y_{ij} := 1 - w_i w_j
\quad\text{for $1 \leq i < j \leq r$.}
\end{equation}
In view of Fact~\ref{fact:lp-simplicial}, Remark~\ref{rem:var-suff},
and Proposition~\ref{prop:suff-rdm},
$$
\Prob{ \text{$F$ is \emph{not} a simplicial face of $\coll{E}_n$} }
	\quad\leq\quad \Prob{ \lspan\{ \vct{w}_1, \dots, \vct{w}_n \} \neq \R^r }
	\ + \ \Prob{ \lspan\{ \smash{\vct{y}_1, \dots, \vct{y}_n} \} \neq \R^{\smash{R'}} }.
$$
\end{remark}

\section{Method 1: The Matrix Chernoff Inequality}
\label{sec:chernoff}

This section contains the proofs of Theorem~\ref{thm:balanced} and~\ref{thm:unbalanced-chernoff}.
The approach relies on a well-known consequence of the matrix Chernoff
inequality~\cite{AW02:Strong-Converse,Tro12:User-Friendly}.

\subsection{Tools}

Let us present a specialization
of the lower tail bound from the matrix Chernoff inequality.
This result was first obtained by Ahlswede \& Winter~\cite[Thm.~19]{AW02:Strong-Converse},
and it later received a significant upgrade~\cite[Thm.~1.1]{Tro12:User-Friendly}.

\begin{fact}[Ahlswede \& Winter; Tropp] \label{fact:chernoff}
Consider a random vector $\vct{x} \in \R^d$
with second-moment matrix $\mtx{\Sigma} := \Expect[ \vct{xx}^\adj]$.
Assume that
$$
\lambda := \lambda_{\min}(\mtx{\Sigma})
\quad\text{and}\quad
\normsq{\vct{x}} \leq B
\quad\text{almost surely.}
$$
Draw iid copies $\vct{x}_1, \dots, \mtx{x}_s$ of the random vector $\vct{x}$.
Then
$$
\Prob{ \lspan\{ \vct{x}_1, \dots, \vct{x}_s \} \neq \R^d }
	\leq d \cdot \exp\left( \frac{-\lambda s}{2B} \right).
$$
\end{fact}

\begin{proof}[Proof Sketch]
This statement follows immediately by applying
the simplified form~\cite[Rem.~5.3]{Tro12:User-Friendly}
of the matrix Chernoff inequality for the minimum eigenvalue
to the random matrices $\mtx{X}_i = \vct{x}_i {\vct{x}_i}^\adj$.
We set the tail parameter $t = 0$.
\end{proof}

\begin{remark}[Refined Probability Bounds]
Fact~\ref{fact:chernoff} gives a good estimate for 
how large $s$ must be to make the probability bound nontrivial.
To obtain more accurate bounds when
$s$ is larger, one must combine the matrix Chernoff
bound with a scalar concentration inequality.
We omit these developments.
\end{remark}

\subsection{Proof of Theorem~\ref{thm:balanced}}

For this result, the balance parameter $p = 0.5$.
We may also assume that $r \geq 2$, or else the result
holds trivially because of Fact~\ref{fact:vertices}.
We instate the notation of Proposition~\ref{prop:suff-rdm}.

Let $\vct{w}_1, \dots, \vct{w}_n$ be iid copies of $\vct{w} \in \R^r$.
Fact~\ref{fact:chernoff} readily implies that
\begin{equation} \label{eqn:w-bd-chernoff}
\Prob{ \lspan\{ \vct{w}_1, \dots, \vct{w}_n \} \neq \R^r }
	\leq r \cdot \exp\left( \frac{- n}{2r} \right).
\end{equation}
Indeed, since $\vct{w} \sim \textsc{SBern}(0.5, r)$,
we quickly determine that $\normsq{\vct{w}} = r$ and that $\Expect[\vct{ww}^\adj] = \Id$.

Now, let $\vct{z}_1, \dots, \vct{z}_n$ be iid copies of $\vct{z} \in \R^R$.
In this case, Fact~\ref{fact:chernoff} yields
\begin{equation} \label{eqn:z-bd-chernoff}
\Prob{ \lspan\{ \vct{z}_1, \dots, \vct{z}_n \} \neq \R^R }
	\leq R \cdot \exp\left( \frac{- n}{2R} \right).
\end{equation}
To establish this point, recall that the random vector $\vct{z} \in \R^R$
is derived from $\vct{w}$ via the formula~\eqref{eqn:z-vec}.
Therefore, $\normsq{\vct{z}} = R$ and a short calculation yields
$\Expect[\vct{zz}^\adj] = \Id$.

To complete the proof, combine the
bounds~\eqref{eqn:w-bd-chernoff} and~\eqref{eqn:z-bd-chernoff}:
$$
\Prob{ \lspan\{ \vct{w}_1, \dots, \vct{w}_n \} \neq \R^r }
	+ \Prob{ \lspan\{ \vct{z}_1, \dots, \vct{z}_n \} \neq \R^R }
	\leq r^2 \cdot \exp\left(\frac{-n}{r^2}\right).
$$
We have used the relations $2r \leq r + R \leq 2R \leq r^2$,
which are valid because $R = 1 + r(r-1)/2$ and $r \geq 2$.
An application of Proposition~\ref{prop:suff-rdm}
completes the proof of Theorem~\ref{thm:balanced}.

\subsection{Proof of Theorem~\ref{thm:unbalanced-chernoff}}

As before, we may assume that $r \geq 2$.
This time, we need the alternative sufficient condition from
Remark~\ref{rem:var-suff-rdm}, and we instate the notation from this remark.
It is also productive to abbreviate $\alpha := (2p-1)^2$,
which is the squared expectation of an $\textsc{SBern}(p)$ random variable.

Let $\vct{w} \sim \textsc{SBern}(p,r)$.  It is immediate that
$\normsq{\vct{w}} = r$.  By direct calculation,
the second-moment matrix of the random vector $\vct{w} \in \R^r$
takes the form
\begin{equation} \label{eqn:Mr}
\mtx{M}_r := \Expect[ \vct{ww}^\adj ] = (1-\alpha) \cdot \Id + \alpha \cdot \mathbf{J} \in \Sym_+^r.
\end{equation}
Recall that $\mathbf{J}$ is the matrix of ones.
It follows immediately that $\lambda_{\min}(\mtx{M}_r) = 1 - \alpha$.
Fact~\ref{fact:chernoff} delivers
\begin{equation} \label{eqn:w-bd-chernoff-p}
\Prob{ \lspan\{ \vct{w}_1, \dots, \vct{w}_n \} \neq \R^r }
	\leq r \cdot \exp\left( - \frac{(1-\alpha) n}{2r} \right),
\end{equation}
where $\vct{w}_1, \dots, \vct{w}_n$ are iid copies of
the random vector $\vct{w}$.

Next, consider the random vector $\vct{y} \in \R^{R'}$,
derived from $\vct{w} \in \R^r$ via the formula~\eqref{eqn:y-vec}.
Note that $\normsq{\smash{\vct{y}}} \leq 4R'$ because the entries of $\vct{y}$
takes values in the set $\{0, 2\}$.  We assert the following bound
on the minimum eigenvalue of the second-moment matrix of $\vct{y}$.

\begin{claim} \label{claim:min-eig}
Let $\mtx{\Sigma} := \Expect[\vct{yy}^\adj]$ be the second-moment matrix of $\mtx{y}$.
Then $\lambda_{\min}(\mtx{\Sigma}) \geq (1-\alpha)^2$.
\end{claim}

\noindent
We will verify Claim~\ref{claim:min-eig} in Section~\ref{sec:min-eig}.
Granted this result, Fact~\ref{fact:chernoff} provides
\begin{equation} \label{eqn:y-bd-chernoff-p}
\Prob{ \lspan\{ \smash{\vct{y}_1, \dots, \vct{y}_n} \} \neq \R^{\smash{R'}} }
	\leq R' \cdot \exp\left( - \frac{(1-\alpha)^2 n}{8R'} \right),
\end{equation}
where $\vct{y}_1, \dots, \vct{y}_n$ are iid copies of the random vector $\vct{y}$.  

Combine~\eqref{eqn:w-bd-chernoff-p} and~\eqref{eqn:y-bd-chernoff-p} to reach
$$
\Prob{ \lspan\{ \vct{w}_1, \dots, \vct{w}_n \} \neq \R^r }
+ \Prob{ \lspan\{ \smash{\vct{y}_1, \dots, \vct{y}_n} \} \neq \R^{\smash{R'}} }
	\leq r^2 \cdot \exp\left( \frac{-(1-\alpha)^2 n}{4r^2} \right).
$$
We have also used the relations $2r \leq r + R' \leq 2R' \leq r^2$, which
hold because $r \geq 2$ and $R' = r(r-1)/2$.
This calculation also depends on the bound $(1-\alpha)^2 \leq (1-\alpha)$,
which is valid because $\alpha \in (0,1)$.  Finally, note that
$1 - \alpha = 4p(1-p)$.
In view of Remark~\ref{rem:var-suff-rdm},
we arrive at Theorem~\ref{thm:unbalanced-chernoff}.

\subsection{Proof of Claim~\ref{claim:min-eig}}
\label{sec:min-eig}

The argument is expressed most easily in the language of multilinear algebra;
see~\cite[Chap.~I]{Bha97:Matrix-Analysis} for more background.
This approach was inspired by conversations with Richard K\"ung.

Introduce the linear space $V := \R^r \vee \R^r$ of symmetric tensors, equipped with
the real inner product $\ip{\cdot}{\cdot}$.
Let $\vee : \R^r \times \R^r \to V$ be the symmetric bilinear map that constructs
an elementary symmetric tensor $\vct{u} \vee \vct{v} \in V$ from two vectors $\vct{u}, \vct{v} \in \R^r$.
We always have the relation $\vct{u} \vee \vct{v} = \vct{v} \vee \vct{u}$.
The space $V$ admits the orthonormal basis
$\{ \mathbf{e}_i \vee \mathbf{e}_j : 1 \leq i \leq j \leq r \}$,
where the $\mathbf{e}_i$ are the standard basis vectors in $\R^r$.
For linear operators $\mtx{A}, \mtx{B}$ acting on $\R^r$, we can define a linear operator
$\mtx{A} \vee \mtx{B}$ acting on $V$ by the rule
$$
\mtx{A} \vee \mtx{B} : \mathbf{e}_i \vee \mathbf{e}_j \longmapsto (\mtx{A} \mathbf{e}_i) \vee (\mtx{B} \mathbf{e}_j)
\quad\text{for $1 \leq i \leq j \leq r$.}
$$
If $\mtx{A}$ and $\mtx{B}$ are both psd, then $\mtx{A} \vee \mtx{B}$ is a psd operator on $V$.

Consider the subspace $W := \lspan\{ \mathbf{e}_i \vee \mathbf{e}_j : 1 \leq i < j \leq r \}$
of the inner product space $V$.
It is natural to treat the random vector $\vct{y} \in \R^{r(r-1)/2}$ as an element of $W$
by identifying $\ip{ \smash{\vct{y}} }{ \smash{\mathbf{e}_i \vee \mathbf{e}_j} } = y_{ij}$ for $1 \leq i < j \leq r$.
Similarly, $\mtx{\Sigma} = \Expect [ \vct{yy}^\adj ]$
is the linear operator on $W$ given by
$$
\ip{ \mathbf{e}_{i'} \vee \mathbf{e}_{j'} }{ \mtx{\Sigma} (\mathbf{e}_i \vee \mathbf{e}_j) }
	= \Expect \left[ \ip{ \mathbf{e}_{i'} \vee \mathbf{e}_{j'} }{ \vct{y} } \ip{ \vct{y} }{ \mathbf{e}_i \vee \mathbf{e}_j }  \right]
	\qquad\text{where}\qquad
\begin{aligned}
	1 &\leq i < j \leq r; \\
	1 &\leq i' < j' \leq r.
\end{aligned}
$$
Using the definition~\eqref{eqn:y-vec} of $\vct{y}$
and the fact that $\vct{w} \sim \textsc{SBern}(p, r)$, we quickly compute that
$$
\ip{ \mathbf{e}_{i'} \vee \mathbf{e}_{j'} }{ \mtx{\Sigma}  (\mathbf{e}_i \vee \mathbf{e}_j) }
	= \begin{cases} (1 - \alpha)^2 + (1 - \alpha^2), & \text{$i = i'$ and $j = j'$} \\
		(1 - \alpha)^2 + \alpha (1 - \alpha), & \text{$i = i'$ xor $j = j'$} \\
		(1 - \alpha)^2 + \alpha (1 - \alpha), & \text{$i = j'$ xor $j = i'$} \\
		(1 - \alpha)^2, & \text{$i \neq i'$ and $j \neq j'$} \\
	\end{cases}
\qquad\text{where}\qquad
\begin{aligned}
1 &\leq i < j \leq r; \\
1 &\leq i' < j' \leq r.
\end{aligned}
$$
This expression is useful, but it takes more work to expose
the spectral properties of $\mtx{\Sigma}$.

The key idea is to identify the operator $\mtx{\Sigma}$ on $W$
as the restriction of an operator $\check{\mtx{\Sigma}}$ on $V$.
Define
\begin{equation} \label{eqn:sigma-prime}
\check{\mtx{\Sigma}} := (1 - \alpha)^2 \cdot (\Id \vee \Id)
	+ \alpha(1-\alpha) \cdot ( \Id \vee \mathbf{J} + \mathbf{J} \vee \Id)
	+ \tfrac{1}{2}(1 - \alpha)^2 \cdot (\mathbf{J} \vee \mathbf{J}).
\end{equation}
Using the fact that $\mathbf{J} \mathbf{e}_i = \mathbf{1} =  \sum_{k = 1}^r  \mathbf{e}_k$ for each index $1 \leq i \leq r$,
we can check that
$$
\ip{  \mathbf{e}_{i'} \vee \mathbf{e}_{j'} }{ \check{\mtx{\Sigma}} (\mathbf{e}_i \vee \mathbf{e}_j) }
	= \ip{ \mathbf{e}_{i'} \vee \mathbf{e}_{j'} }{ \mtx{\Sigma}( \mathbf{e}_i \vee \mathbf{e}_j) }
	\quad\text{where}\quad
\begin{aligned}
	1 &\leq i < j \leq r; \\
	1 &\leq i' < j' \leq r.
\end{aligned}
$$
As promised, $\mtx{\Sigma}$ is the restriction of $\check{\mtx{\Sigma}}$ to $W$.

The representation~\eqref{eqn:sigma-prime} of the operator $\check{\mtx{\Sigma}}$
allows us to determine its spectrum with ease.
The operators $\Id$ and $\mathbf{J}$ are both psd,
so the operators $\Id \vee \mathbf{J}$ and $\mathbf{J} \vee \Id$
and $\mathbf{J} \vee \mathbf{J}$ are also psd.
Weyl's monotonicity principle implies that
$$
\lambda_{\min}(\check{\mtx{\Sigma}}) \geq \lambda_{\min}\big( (1 - \alpha)^2 \cdot (\Id \vee \Id) \big)
	= (1 - \alpha)^2.
$$
The first inequality depends on the fact that $\alpha \in [0, 1]$,
and the second relation holds because $\Id \vee \Id$ is the identity operator on $V$.
Finally, the operator $\mtx{\Sigma}$ is a restriction of $\check{\mtx{\Sigma}}$,
so we conclude that
$$
\lambda_{\min}( \mtx{\Sigma} ) \geq \lambda_{\min}( \check{\mtx{\Sigma}} )
	\geq (1 - \alpha)^2.
$$
This establishes Claim~\ref{claim:min-eig}.

\section{Method 2: Oliveira's Lower Tail Inequality}
\label{sec:oliveira}

This section contains the proof of Theorem~\ref{thm:unbalanced}.
The argument depends on a recent matrix concentration inequality
due to Oliveira~\cite{Oli16:Lower-Tail}.

\subsection{Tools}

We begin with a summary of the technical tools that we require.
The key result is a specialization of Oliveira's lower tail
inequality~\cite[Thm.~1.1]{Oli16:Lower-Tail}.

\begin{fact}[Oliveira] \label{fact:oliveira}
Consider a random vector $\vct{x} \in \R^d$
whose second-moment matrix $\Expect[ \vct{xx}^\adj]$ is nonsingular.
Compute the hypercontractive parameter
\begin{equation} \label{eqn:weak-variance}
h := h(\vct{x}) := \max_{\vct{u} \neq \vct{0}}
	\frac{\Expect (\vct{x}^\adj \vct{u})^4}{\big[ \Expect (\vct{x}^\adj \vct{u})^2 \big]^2}.
\end{equation}
Draw iid copies $\vct{x}_1, \dots, \mtx{x}_s$ of the random vector $\vct{x}$.
Then
$$
\Prob{ \lspan\{ \vct{x}_1, \dots, \vct{x}_s \} \neq \R^d }
	\leq 2 \exp\left( \frac{d}{2} - \frac{s}{162 h} \right).
$$
\end{fact}

To bound the parameter $h$ that appears in the last result,
we need the following hypercontractive inequality.
For example, see~\cite[Thm.~10.21]{ODo14:Analysis-Boolean}.

\begin{fact}[Hypercontractivity] \label{fact:hyper}
Consider a polynomial $q : \{ \pm 1 \}^n \to \R$ with real coefficients and degree $k$.
Draw a random vector $\vct{\xi} \sim \textsc{SBern}(p,n)$ where $p \in (0,1)$. Then the random variable $X : = q(\vct{\xi})$ satisfies
$$
\frac{\Expect X^4}{\big(\Expect X^2\big)^{2}}
	\leq \left[ \frac{9}{p(1-p)} \right]^{k}. $$
\end{fact}

\subsection{Proof of Theorem~\ref{thm:unbalanced}}
\label{sec:pf-balanced}

As usual, assume that $r \geq 2$.
We also instate the notation from Proposition~\ref{prop:suff-rdm}.
Recall that the random vector $\vct{w} \sim \textsc{SBern}(p,r)$.
It is easy to verify that the second-moment matrix $\Expect[ \vct{ww}^\adj ]$
is nonsingular; see~\eqref{eqn:Mr}.
The entries of the random vector $\vct{z} \in \R^R$ are derived from $\vct{w}$
by means of the formula~\eqref{eqn:z-vec}.
The second-moment matrix $\Expect[ \vct{zz}^\adj ]$ is also nonsingular;
indeed, for fixed $\vct{u} \in \R^R$, the random variable $(\vct{z}^\adj \vct{u})^2$
is identically zero only if $\vct{u} = \vct{0}$.

Let us begin with the probability that $n$ iid copies $\vct{w}_1, \dots, \vct{w}_n$
of the random vector $\vct{w} \in \R^r$ span all of $\R^r$.
We will use Oliveira's result, Fact~\ref{fact:oliveira}, to establish that
\begin{equation} \label{eqn:w-bd-oliveira}
\Prob{ \lspan\{\vct{w}_1, \dots, \vct{w}_n \} \neq \R^r }
	\leq 2 \exp\left( \frac{r}{2} - \frac{p(1-p) n}{1458} \right).
\end{equation}
Observe that, for any vector $\vct{u} \in \R^r$, the linear form
$\vct{w}^\adj \vct{u} = \sum_{i=1}^r w_i u_i$
is a polynomial of degree one in the entries of $\vct{w}$.
Fact~\ref{fact:hyper} implies that
$$
\frac{\Expect (\vct{w}^\adj \vct{u})^4}{\big[\Expect (\vct{w}^\adj \vct{u}^2)\big]^2}
	\leq \frac{9}{p(1-p)}.
$$
Therefore, the hypercontractive parameter $h(\vct{w}) \leq 9 p^{-1} (1-p)^{-1}$.
The claim~\eqref{eqn:w-bd-oliveira} now follows from Fact~\ref{fact:oliveira}.

Second, we study the probability that iid copies $\vct{z}_1, \dots, \vct{z}_n$
of the random vector $\vct{z} \in \R^R$ span all of $\R^R$.
We will apply Fact~\ref{fact:oliveira} to obtain
\begin{equation} \label{eqn:z-bd-oliveira}
\Prob{ \lspan\{\vct{z}_1, \dots, \vct{z}_n\} \neq \R^R }
	\leq 2 \exp\left( \frac{R}{2} - \frac{p^2(1-p)^2 n}{13122} \right).
\end{equation}
For any vector $\vct{u} \in \R^R$, we can express
$\vct{z}^\adj \vct{u} = u_0 + \sum_{i < j} w_i w_j u_{ij}$.
This is a polynomial of degree two in the entries of $\vct{w}$.
Fact~\ref{fact:hyper} implies that
$$
\frac{\Expect (\vct{z}^\adj \vct{u})^4}{\big[ \Expect(\vct{z}^\adj \vct{u})^2 \big]^2}
	\leq \left[ \frac{9}{p(1-p)} \right]^{2}.
$$
Therefore, the parameter $h(\vct{z}) \leq 81 p^{-2} (1-p)^{-2}$,
and the claim~\eqref{eqn:z-bd-oliveira} follows from Fact~\ref{fact:oliveira}.

To complete the argument,
combine the inequalities~\eqref{eqn:w-bd-oliveira} and~\eqref{eqn:z-bd-oliveira}
to arrive at the estimate
$$
\Prob{ \lspan\{\vct{w}_1, \dots, \vct{w}_n \} \neq \R^r }
	+ \Prob{ \lspan\{\vct{z}_1, \dots, \vct{z}_n\} \neq \R^R }
	\leq 4 \exp\left( \frac{r^2}{4} - \frac{p^2(1-p)^2 n}{13122} \right).
$$
We have used the bounds $r \leq R \leq r^2/2$,
which are valid because $r \geq 2$ and $R = 1 + r(r-1)/2$.
Introduce this inequality into Proposition~\ref{prop:suff-rdm}
to complete the proof of Theorem~\ref{thm:unbalanced}.

\begin{remark}[Alternative Proof]
Let us mention an alternative approach to Theorem~\ref{thm:unbalanced} based
on Mendelson's Small Ball Method~\cite{KM15:Bounding-Smallest,Tro15:Convex-Recovery}.
This technique yields a comparable outcome, but it takes more steps to apply.
\end{remark}

\section*{Acknowledgments}

The author thanks Richard K{\"u}ng and Benjamin Recht for helpful conversations related to this work.
This research was partially supported by
ONR award N00014-11-1002 and the Gordon \& Betty Moore Foundation.

\bibliographystyle{myalpha}

\end{document}